\documentclass{amsart}
              [1997/02/02 v1.2e PROC Author Class]

\copyrightinfo{2006}
  {American Mathematical Society}

\newtheorem{theorem}{Theorem}[section]

\newtheorem{lemma}[theorem]{Lemma}

\newcommand{\omf}{\omega_f}
\newcommand{\df}{D(f)}
\newcommand{\fxr}{f:X\to \bf R}
\newcommand{\ven}{\varepsilon}

\begin{document}

\title[Characterization of sets of discontinuity points]{A characterization of sets of discontinuity points of separately continuous
functions of many variables on products of metrizable spaces}

\author{V.K.Maslyuchenko}
\address{Department of Mathematics\\
Chernivtsi National University\\ str. Kotsjubyn'skogo 2,
Chernivtsi, 58012 Ukraine}
\email{mathan@chnu.cv.ua}

\author{V.V.Mykhaylyuk}
\address{Department of Mathematics\\
Chernivtsi National University\\ str. Kotsjubyn'skogo 2,
Chernivtsi, 58012 Ukraine}
\email{vmykhaylyuk@ukr.net}

\subjclass[2000]{Primary 54C08; secondary 54C30, 54C05}


\commby{Ronald A. Fintushel}


\keywords{separately continuous functions, meager set, discontinuity points set}

\begin{abstract}
It is shown that a set in product  of $n$ metrizable spaces is the
discontinuity points set of some separately continuous function
if and only if this set can be represented as the union of
a sequence of $F_{\sigma}$-sets which are locally projectively meager.
\end{abstract}

\maketitle
\section{Introduction}

The general inverse problem of the separately continuous mappings theory is as follows: for a given set $E$ which contained in the product
$X$ of topological spaces $X_1,...,X_n$ to construct a separately continuous function $f:X\to \bf R$ with the discontinuity points set $D(f)=E$. Direct theorems on the size of the set $D(f)$ for separately continuous functions give necessary conditions on the set $E$. Solution of the inverse problem for a given necessary conditions gives a characterization of the discontinuity points set of separately continuous functions defined on the product of spaces from some classes.

There are few results which give such characterization. R.Kershner in [1] gives the characterization of set $D(f)$ for separately continuous functions $f: {\bf R}^n \to \bf R$. For $n=2$ this characterization has such view: a set $E\subseteq {\bf R}^2$ is the discontinuity points set for some separately continuous function $f:{\bf R}^2\to \bf R$ if and only if $E$ is a $F_{\sigma}$-set and its projections on the both axis are meagre, that is $E$ is a projectively meagre set. J.~C.~Breckenridge and T.~Nishiura in [2] solved the inverse problem for projectively meager $F_{\sigma}$-sets which contained in the product od two metrizable spaces. According to Calbrix-Troallic theorem from [3], for spaces $X$ and $Y$ with the second countable axiom and for a separately continuous function $f:X\times Y \to \bf R$ the set $D(f)$ is projectively meager. Since for every real-valued function $f$ the set $D(f)$ is a $F_{\sigma}$-set, Breckenridge-Nishiura theorem and Calbrix-Troallic theorem a characterization of discontinuity points set for separately continuous functions defined on the product of two separable metrizable spaces, in particular, on the product of two metrizable compacts.

Some new approaches to the solution of the inverse problem on the product of metrizable spaces were proposed in [4,5,6]. The method from [6] used locally finite systems and Stone's theorem on the paracompactness of a metrizable spaces. This method was developed in [7] (see also [8]). Moreover, it was obtained in [9] using a theorem on the dependence of separately continuous functions on at most countable coordinates a characterization of discontinuity points set for separately continuous functions $f:X\times Y \to \bf R$, where $X$ and $Y$ are the products of metrizable compacts.

A similar to Keshner's characterization is not true for separately continuous functions on the products of metrizable spaces. It was constructed in [10] an example of separately continuous function $f:{\bf R}\times l_{\infty} \to \bf R$ for which the projection of $D(f)$ on the first multiplier coincides with $\mathbb R$ (see also [8,11]). Although in this example the projection of $D(f)$ is not meager,  it was locally projectively one-pointed. This shows that the characterization of $D(f)$ for separately continuous functions $f:X\times Y\to \bf R$ defined on the product of metrizable spaces can be formulated in a locally terms. Such approach was realized in [7]. It was obtained there the characterization of $D(f)$ for separately continuous functions $f:X\times
Y\to \bf R$ where $X$ and $Y$ are metrizable (see [12] and [8]).

In this paper we consider the case of $n$ variables functions $f:X_1\times X_n\to\mathbb R$. We obtain the following result: if $X_1,...,X_n$ are metrizable speces and $X=X_1\times ...\times X_n$, then for set $E\subseteq X$ there exists a separately continuous function $f:X\to \bf R$ with $D(f)=E$ if and only if $E$ is the union of a sequence of $F_{\sigma}$-sets which are locally projectively meager. To prove the necessity we use the quasicontinuity and Banach Theorem on Category [13, p.87-90]. In the proof of the sufficiency we use an approximation method from [6,7,8] which based on the notion of favorable pair.

\section{Definitions and auxiliary statements}
For a function $\fxr$ and set $A\subseteq X$ the real $$\omf(A)=\sup
\{|f(x')-f(x'')|:x',x''\in A\}$$ is called {\it the oscillation of $f$ on $A$}. If $X$ is a topological space, $x_0\in X$ and ${\mathcal U}_{x_0}$ is the system of all neighborhoods of $x_0$ in $X$, then $$\omf(x_0)=\inf \{ \omf (U):U\in {\mathcal U}_{x_0} \}$$ is {\it the oscillation of  $f$ at $x_0$}.
The function $\omf :X\to [ 0,+\infty ]$ is upper semicontinuousw and the sets $\{ x\in X: \omf (x) \geq \ven \}$ are closed for every
$\ven >0$.

A function $\fxr$ is called {\it quasicontinuity}, if for every $x\in X$ and every its neighborhood $U$ and every $\ven >0$ there exists nonempty open set $V$ in $X$ such that $V\subseteq U$ i $|f(v)-f(x)|<\ven $ for every $v\in V$.
The next lemma follows from definition.

\begin{lemma}\label{l:2.1} Let $X$ be a topological spaces. A function $\fxr$ is quasicontinuous if and only if $f(G)\subseteq
\overline {f(A)}$ for every open in $X$ set $G$ and every dense in $G$ set $A\subseteq X$.
\end{lemma}

Let $X=X_1\times ...\times X_n$ be a topological product of $n$ spaces, $x=(x_1,...,x_n)\in X$ and $k\in \{1,...,n\}$. We put
$$\hat{X}_k=X_1\times ...\times X_{k-1}\times X_{k+1}\times ...\times X_n,$$
$\hat{x}_k=(x_1,...,x_{k-1},x_{k+1},...,x_n)$ and $p_{\hat{X}_{k}}(x)=\hat{x}_k$. The mapping $p_{\hat{X}_{k}}:X\to \hat{X}_k$ is called by {\it $k$-th  projection}. We say that a set $A\subseteq X$ is {\it projectively nowhere dense (meager) with respect to $k$-th variable}, if $p_{\hat{X}_{k}}(A)$ is nowhere dense (meager) in $\hat{X}_k$. If for every $x\in X$ there exists a neighborhood $U$ of $x$ in $X$ such that $A\bigcap U$ is projectively nowhere dense (meager) with respect to $k$-th variable, then $A$ is called {\it locally projectivelly nowhere dense (meagre) with respect to $k$-th variable}. Projectively nowhere dense (meagre) with respect to every variable set is called {\it projectively nowhere dense (meagre)}. The same concerns to the local versions of these notions.

Function $f:X_1\times ...\times X_n\to \bf R$ is called {\it separately continuous}, if it continuous with respect to each variable. To solve the inverse problem we need the following proposition which generalize Lemma 3.1.4 from [8].

\begin{lemma}\label{l:2.2} Let $X=X_1\times ...\times X_n$ be a topological spaces of $n$ spaces, $\mathcal W$ be a locally finite system of open in $X$ nonempty set and $(f_ W)_{W\in {\mathcal W}}$ be a family of lower semicontinuous separately continuous functions $f_W:X\to \bf R$ such that $f_W(x)=0$ on $X\setminus W$ for every $W\in {\mathcal W}$. Then the function $f(x)=\sum \limits_{W\in \mathcal{W}}f_W(x)$ is lower semicontinuous and $\df =\bigcup \limits_{W\in \mathcal{W}}D(f_W)$.
\end{lemma}

\begin{proof} Let $x_0\in X $, $U$ be an open neighborhood of $x_0$ in $X$ such that the system ${\mathcal W}_U=\{W\in {\mathcal W}: W\bigcap U \not =\emptyset \}$ is finite and $g=\sum \limits_{W \in {\mathcal W}_U}f_W$. The function $g$ is lower semicontinuous and separately continuous at $x_0$ as a finite sum of such functions. The restriction of $f$ on $U$ coincides with $g$. Thus $f$ is the same. Moreover, if
$x_0 \not \in \bigcup \limits_{W \in {\mathcal W}}D(f_W)$, then all functions $f_W$ are continuous at $x_0$. Hence $g$ and $f$ are continuous at $x_0$ and $x_0 \not \in \df$. Thus, $\df \subseteq \bigcup \limits_{W \in {\mathcal W}}D(f_W)$.

To prove the inverse inclusion we fix $W_0 \in {\mathcal W}$ and consider the function $$g_{W_0}=\sum \limits_{W \in {\mathcal W}\setminus
\{W_0\}} f_W.$$ The function $g_{W_0}$ is lower semicontinuous. Thus, the function $f_{W_0}$, i $f=g_{W_0} + f_{W_0}$ is lower semicontinuous too. This implies (see, for example, Lemma 2.2.1 from [8]) that $\df=D(g_{W_0})\bigcup D(f_{W_0})$. Hence, $D(f_{W_0}) \subseteq \df$.
\end{proof}

\section{Discontinuity points set of $KC$-functions}

The following proposition plays an important role in the prove of the main result and generalize Theorem 3.2.3 from [8].

\begin{theorem}\label{th:3.1} Let $X$ be a baire spase, $Y$ be a metric space and a function $f:X\times Y \to \bf R$ be a function which is quasicontinuous with respect to the first variable and continuous with respect to the second variable. Then there exists a sequence of $F_{\sigma}$-sets $E_n$ such that $\df= \bigcup \limits_{n=1}^{\infty}E_n$ and for every open in $Y$ set $V$ with ${\rm diam} V< \frac{1}{n}$ the set $E_n \bigcap(X\times V)$ is projectively meagre with respect to the second variable.
\end{theorem}

\begin{proof} For $(x,y)\in X\times Y$ we put $f^x(y)=f_y(x)=f(x,y)$. By $V(y_0,\ven)$ we denote the open bull in $Y$ with the center $y_0$ and the radius $\ven$. For $m, n\in\mathbb N$ we consider the sets
$$
E_{mn}=\{(x,y)\in X\times Y: \omf(x,y)\geq \frac{1}{m} \quad \mbox{and} \quad \omega_{f^x}(V(y,\frac{1}{n}))<\frac{1}{4m}\},
$$
$$
F_m=\{ (x,y)\in X\times Y: \omf(x,y) \geq \frac{1}{m}\}, F_{mn}=\overline{E_{mn}} \quad  \mbox{and} \quad  E_n=\bigcup \limits_{m=1}^{\infty}F_{mn}.
$$
The sets $F_m$ are closed. It follows from the continuity of $f^x$ that $\bigcup \limits_{n=1}^{\infty}E_{mn}=F_m$ for every $m$. Moreover, $\df=\bigcup \limits_{m=1}^{\infty}F_m$. This implies that $\bigcup \limits_{n=1}^{\infty}E_n=\df$.

Let $m,n\in \bf N$ and $V$ be an open in $Y$ set with ${\rm diam} V <\frac{1}{n}$. We show that the set $M_0=E_{mn}\bigcap (X\times
V)$ is projectively meagre with respect to the second variable. Assume the contrary. We consider the projection mapping $p_X=p_{\hat Y}$ on the first multiplier and the set $A_0=p_X(M_0)$. According to the assumption, there exists nonempty open in $X$ set $U_0$ such that $\overline{A_0}\supseteq U_0$. Fix any point $y_0\in V$. It follows from the quasicontinuity of $f_{y_0}$ that there exists an open in $X$ nonempty set $U$ such that $U\subseteq U_0$ and $\omega_{f_{y_0}}(U)<\frac{1}{4m}$. Clearly that the set $A=U\bigcap A_0$ is dense in $U$. Since $A\subseteq p_X(M_0)$, for every $a\in A$ there exists $v_a\in V$ such that $(a,v_a)\in E_{mn}$, in particular, $\omega_{f^a}(V(v_a,\frac{1}{n}))<\frac{1}{4m}$. But ${\rm diam}V<\frac{1}{n}$, thus, $V\subseteq V(v_a,\frac{1}{n})$ and $\omega_{f^a}(V)<\frac{1}{4m}$ for every $a\in A$. Let $a_1,a_2\in A$ i $y_1,y_2\in V$. Then
$$
|f(a_1,y_1)-f(a_2,y_2)|=|f(a_1,y_1)-f(a_1,y_0)+f(a_1,y_0)-f(a_2,y_0)+f(a_2,y_0)-
$$
$$
-f(a_2,y_2)|\leq|f^{a_1}(y_1)-f^{a_1}(y_0)|+|f_{y_0}(a_1)-f_{y_0}(a_2)|+|f^{a_2}(y_0)-f^{a_2}(y_1)|<\frac{3}{4m},
$$
thus, $\omf(A\times V)\leq \frac{3}{4m}$. It is wellknown  (see [2, Theorem 2.2]) that $f$ is quasicontinuous. Since the set $A\times V$ is dense in $U\times V$, it follows from Lemma \ref{l:2.1} that $f(U\times V)\subseteq \overline{f(A\times V)}$. Therefore $\omf(U\times V)=diamf(U\times V)\leq diam \overline{f(A\times V)}= diam f(A\times V)= \omf(A\times V)$. Thus $\omf(U\times V)\leq \frac{3}{4m}<\frac{1}{m}$. Now we have
$E_{mn}\bigcap (U\times V)= \emptyset$, hence $A=\emptyset$, a contradiction.

Thus the set $A_0=p_X(M_0)$ is nowhere dense in $X$. Then the closure $\overline {A_0}$ is nowhere dense in $X$ set too. Since $X\times V$ is open in $X\times Y$, we have $$F_{mn} \bigcap (X\times V)=\overline {E_{mn}}\bigcap (X\times
V)\subseteq \overline{E_{mn}\bigcap (X\times V)}=\overline M_0.$$
It follows from the continuity of $p_X$ that $p_X(\overline{M_0})\subseteq \overline{p_X(M_0)}=\overline
{A_0}$. Thus the set $F_{mn}\bigcap (X\times V)$ is projectively meagre with respect to the second variable. Therefore the set
$$E_{n}\bigcap (X\times V)= \bigcup \limits_{m=1}^{\infty}F_{mn}\bigcap (X\times V)$$ is projectively meagre with respect to the second variable. It remains to note that every $E_n$ is a $F_{\sigma}$-set.
\end{proof}

\section{Main result}

\begin{theorem} Let $X_1,...,X_n$ are metrizable spaces, $X=X_1\times ... \times X_n$ and $E\subseteq X$. Then there exists a separately continuous function $\fxr$ with $\df =E$ if and only if the set $E$ is the union of a sequence of a $F_{\sigma}$-sets which are locally projectively meagre.
\end{theorem}

\begin{proof} Necessity. Let $\fxr$ be a separately continuous function such that $\df=E$. Fix an index $k=1,...,n$ and prove that $E$ is the union of a sequence of a $F_{\sigma}$-sets $E_{m}^{(k)}$ which are locally projectively meagre with re4spect to the $k$-th variable. According to the Banach Theorem on Category
[13, p.87-90], the space $\hat{X}_k$ is the disjoint union of an open set $G$ which is a Baire subspace of $\hat{X}_k$, an open meagre set $H$ and a closed nowhere dense set $\Gamma$. The space $\hat{X}_k$ is metrizable as a finite product of metrizable spaces. Hence, $\hat{X}_k$ is perfect. Thus, open in $\hat{X}_k$ set $H$ is a $F_{\sigma}$-set in $\hat{X}_k$. Since set $E$ is a $F_{\sigma}$-set in $X$ and the projection mapping $p_{\hat{X}_k}:X\to \hat{X}_k$ is continuous, the set
$$E_0^{(k)}=E\cap p^{-1}_{\hat{X}_k}(H\cup \Gamma)$$
is a $F_{\sigma}$-set in $X$. Moreover, $p_{\hat{X}_k}(E_0^{(k)})\subseteq H\bigcup \Gamma$, thus, the set $E_0^{(k)}$ is projectively meagre with respect to the $k$-th variable.

The subspace $\tilde{X}=p^{-1}_{\hat{X}_k}(G)$ of $X$ is homeomorphic to the product $G\times {X}_k$ and we can to consider the restriction $g=f_{|\tilde{X}}$ as a function of two variables $\hat{x}_k$ and $x_k$. Since $\tilde X$ is an open subspace of $X$, $D(g)=\df \bigcap \tilde X$. Clearly that $g$ is continuous with respect to the second variable because $f$ is separately continuous.

We show that $g$ is quasicontinuous with respect to the first variable. Let $\hat{x}_k\in G$. Since $G$ is open in $\hat{X}_k$, for every $i=1,...,n, i\not = k$ there exist open in $X_i$ sets $U_i$ such that for their product $U$ we have $\hat{x}_k\in U\subseteq G$. Since $G$ is a Baire space, $U$ is a Baire space too. Thus, every multiplier $U_i$ is Baire. Fix $x_k\in X_k$ and consider the restriction $h$ of the function $g_{x_k}:G\to \bf R$ on the set $U$. The function $h:U\to \bf R$ is a separately continuous function of $n-1$ variables which defined on the product of Baire metrizable spaces $U_i$. Since this product is Baire, according to Theorem 2.4 from [2] the function $h$ is quasicontinuous. Therefore $g_{x_k}$ is quasicontinuous at $\hat{x}_k$. Thus, $g$ is quasicontinuous with respect to the first variable.

Now using the Theorem \ref{th:3.1} to the function $g$ we obtain that
$$D(g)=\bigcup \limits _{m=1}^{\infty}E_m^{(k)},$$
where $E_m^{(k)}$ are $F_{\sigma}$-sets which are locally projectively meagre with respect to the second variable in $\tilde X$. Since $\tilde X$ is an open subspace of the metrizable space $X$, the set $E_m^{(k)}$ is an $F_{\sigma}$-set in $X$ and it is an locally projectively meagre set with respect to the $k$-th variable. But $E=E_0^{(k)}\bigcup D(g)$. Thus, $E=\bigcup \limits _{m=0}^{\infty} E_m^{(k)}$ is a needed representation.

Now we put
$$E_m=\bigcap \limits_{k=1}^n\bigcup \limits
_{j=0}^{m} E_j^{(k)}.$$
Since $\bigcup \limits _{j=0}^{\infty}E_j^{(k)}=E$ for every $k=1,...,n$, $\bigcup \limits
_{m=1}^{\infty} E_m=E$. Moreover, the set $\bigcup \limits_{j=0}^{m} E_j^{(k)}$ is locally projectively meagre with respect to the $k$-th variable and is a $F_{\sigma}$-set for every $k=1,...,n$. Therefore, their intersection $E_m$ is a locally projectively meagre $F_{\sigma}$-set.

Sufficiency. Let $E=\bigcup \limits_{m=1}^{\infty} E_m$, where $E_m$ are locally projectively meagre $F_{\sigma}$-sets. We show firstly that $E$ is the union of a sequence of locally projectively meagre closed sets. It is sufficient to show that every set $E_m$ has the same view. It follows from Stone's theorem on paracompactness of a metrizable space [14,ñ.414] that there exists an open locally finite cover ${\mathcal U}$ of $X$ such that for every $U\in {\mathcal U}$ the set $M=E_m\bigcap U$ is projectively meagre. Since $M$ is a $F_{\sigma}$-set, $M=\bigcup \limits _{j=1}^{\infty} B_{0,j}$, where $(B_{0,j})_{j=1}^{\infty}$ is an increasing sequence of closed sets. For every $k=1,...,n$ the projection $M_k=p_{\hat{X}_k}(M)$ is a meagre set in $\hat{X}_k$. Therefore there exists an increasing sequence of closed nowhere dense in $\hat{X}_k$ sets $B_{k,j}$ such that $M_k\subseteq \bigcup \limits_{j=1}^{\infty}B_{k,j}$. Now we put
$$B_j^{(U)}= B_{0,j}\bigcap (\bigcap \limits_{k=1}^{n} p^{-1}_{\hat{X}_k}(B_{k,j})).$$ Clearly that sets $B_j^{(U)}$ are closed and projectively nowhere dense. Moreover, $M=\bigcup \limits_{j=1}^{\infty}B_j^{(U)}$. Really, if $x\in M$, then there exist indexes $j_k$, $k=0,1,...,n$ such that $x\in B_{0,j_0}$ and $p_{\hat{X}_k}(x)\in B_{k,j_k}$ for $k=1,...,n$. Then for $j=max\{j_0,...,j_n\}$ we have $x\in B_{0,j}$ and $p_{\hat{X}_k}(x)\in B_{k,j}$ for $k=1,...,n$. Thus, $x\in B_{j}^{(U)}$. Since the family $(B_{j}^{(U)})_{U\in \mathcal{U}}$ is locally finite, the set $B_j=\bigcup \limits_{U\in \mathcal {U}}B_{j}^{(U)}$ is a closed locally projectively nowhere dense set. Moreover,
$$\bigcup \limits_{j=1}^{\infty}B_{j}=\bigcup \limits_{U\in \mathcal
{U}}(\bigcup \limits_{j=1}^{\infty}B_{j}^{(U)})= \bigcup
\limits_{U\in \mathcal {U}}(E_m\bigcap U)=E_m,$$,
because $\mathcal {U}$ is a cover of $X$.

Thus, $E=\bigcup \limits_{m=1}^{\infty}F_m$, where $F_m$ are closed locally projectively nowhere dense sets. We fix some $m\in \bf N$ and show that there exists a separately continuous lower semicontinuous function $f_m:X\to [0,1]$ such that $D(f_m)=F_m$. Let $\mathcal U$ be a locally finite open cover of $X$ such that for every ${U\in \mathcal {U}}$ the set $F_m\bigcap U$ is projectively nowhere dense. We fix ${U\in \mathcal {U}}$ and put $L=F_m\bigcap U$. The sets $L_k=\overline{p_{\hat{X}_k}(L)}$ are closed and nowhere dense in $\hat{X}_k$. Thus set $A=\bigcup \limits^{n}_{k=1}p_{\hat{X}_k}^{-1}(L_k)$ is closed and nowhere dense in $X$. Denote by $\mathcal T$ the product topology on $X$. According to Theorem 3.2.1 from [8] the pair $(X,A)$ is favorable, that is there exist sequences of families $\tau_i:A\to \mathcal T$ of open in $X$ sets and functions $\pi_i:A\to X$ such that

(1) systems ${\mathcal W}_i=\tau_i(A)=\{\tau_i(a):a\in A \}$ are locally finite;

(2) $\pi_i(a)\in \tau_i(a)$ and $\pi_i(a)=\pi_i(b)$, if
$\tau_i(a)=\tau_i(b)$;

(3) $\tau_i(a)\bigcap A=\emptyset$;

(4) $\lim \limits_{i\to \infty} \pi_i(a)=a$;

(5) for every neighburhood $U_1$ of arbitrary point $x_0$ in $X$ there exists neighborhood $U_2$ of $x_0$ in $X$ such that for every integer $i_0$
the intersection $\tau_j(x)\bigcap U_2=\emptyset$ for all $j>i_0$ and $x\in A\setminus U_1$

for every integer $i$ and arbitrary points $a,b\in A$.

According to $(2)$ for every $W=\tau_i(a)\in {\mathcal W}_i$ the point $\pi_i(a)\in \tau_i(a)$ does not depend on $a$ and depends on $W$ only. This point we denote by $p_i(W)$.

We put $F=\overline {L}$. Since $L\subseteq A$ and $A$ is closed, $F\subseteq A$. We construct a separately continuous lower semicontinuous function $g_m^{(U)}:X\to [0,+\infty)$ for which $D(g_m^{(U)})=F$. Denote ${\mathcal V}_i=\tau_i(F)=\{\tau_i(a):a\in F\}$. For every $V \in
{\mathcal V}_i$ we choose a continuous function $\varphi_{V,i}:X\to [0,1]$ such that $\varphi_{V,i}(p_i(V))=1$ and
$\varphi_{V,i}(x)=0$ on $X\setminus V$. Put
$$
g_{m,i}(x)=\sum \limits_{V\in {\mathcal V}_i}\varphi_{V,i}(x) \qquad \mbox{and} \qquad g_{m}^{(U)}(x)=\sum \limits_{i=1}^{\infty}g_{m,i}(x),
$$
for every $x\in X$. Since the systems ${\mathcal V}_i$ are locally finite, the functions $g_{m,i}$ are continuous. Moreover, $g_{m,i}(x)\geq 0$ for every $x\in X$.

Let $x_0 \in X\setminus F$. We show that the function $g_{m}^{(U)}$ is correctly defined and continuous at $x_0$. Choose a neighborhood $U_1$ of $x_0$ such that $U_1\bigcap F=\emptyset$. According to (5) we choose a neighborhood $U_2$ of $x_0$ and an integer $i_0$ such that $\tau_i(a)\bigcap U_2 = \emptyset$ for all $i>i_0$ and $a\in A\setminus U_1$. But $F\subseteq A\setminus U_1$, therefore $\varphi_{V,i}(x)=0$ for all $x\in U_2$, $i\geq i_0$ and $V\in {\mathcal V}_i$. Thus, $g_{m,i}(x)=0$
on $U_2$ for $i>i_0$. Then $g_{m}^{(U)}(x)=\sum\limits_{i=1}^{i_0}g_{m,i}(x)$ on $U_2$. This implies that $g_{m}^{(U)}$ is correctly defined and continuous at $x_0$.

Let $x_0\in F$. Since $F=\overline{L}$, for every $k=1,...,n$ we have $\hat{x}_{0,k}=p_{\hat{X}_k}(x_0)\in
p_{\hat{X}_k}(\overline{L})\subseteq \overline{p_{\hat{X}_k}(L)}=L_k$. Therefore, $p_{\hat{X}_k}^{-1}(\hat{x}_{0,k})\subseteq A$. It follows from $(3)$ that $g_m^{(U)}(x)=0$ on $A$. Thus, $g_m^{(U)}$ is separately continuous at $x_0$. Moreover, $g_m^{(U)}$ is lower semicontinuous at $x_0$, because $g_m^{(U)}(x_0)=0$ and $g_m^{(U)}(x)\geq 0$ on $X$. For every $i$ we have $V_i=\tau_i(x_0)\in {\mathcal V}_i$ and
$$g_m^{(U)}(z_i)\geq g_{m,i}(z_i)\geq \varphi_{V_i,i}(z_i)=1$$ where $z_i=p_i(V_i)=\pi_i(x_0)$. Note that $\lim \limits_{i\to \infty}z_i=x_0$ according to $(4)$. Therefore $x_0\in D(g_m^{(U)})$.

Since $X$ is perfectly normal, for every $U\in \mathcal U$ there exists a continuous function $h_U:X\to [0,1]$ such that $U=h_U^{-1}((0,1])$. The functions $g_{U,m}=h_U\cdot g_m^{(U)}$ are separately continuous nonnegative lower semicontinuous and $g_{U,m}(x)=0$ on $X\setminus U$. Then by Lemma \ref{l:2.2} the function  $g_{m}=\sum \limits_{U\in \mathcal {U}}g_{U,m}$ is separately continuous nonnegative and lower semicontinuous on $X$. Moreover $D(g_m)=\bigcup \limits_{U\in \mathcal {U}}D(g_{U,m})$. Clearly that $D(g_{U,m})\subseteq D(g_{m}^{(U)})=\overline{F_m\bigcap U}\subseteq F_m$. Therefore $D(g_m)\subseteq F_m$. Let $x_0\in F_m$.
Since $\mathcal U$ covers $X$, there exists $U_0\in \mathcal U$ such that $x_0\in U_0$. Then $h_{U_0}(x)>0$ on $U_0$. Thus,
$D(g_m^{(U_0)})\bigcap U_0\subseteq D(g_{U_0,m})$. Now we have $x_0\in F_m\bigcap U_0\subseteq D(g_{U_0,m})\subseteq D(g_m)$.
Hence $D(g_m)=F_m$. Let $\psi:[0,+\infty)\to [0,1)$ be a homeomorphism. It remains to put $f_m=\psi \circ g_m$.

Now we consider the function
$$f=\sum\limits_{m=1}^{\infty}\frac{1}{2^m} f_m.$$ According to Corollary 2.2.2 from [8], we obtain $\df=\bigcup \limits_{m=1}^{\infty}D(f_m)=E$,
besides $f$ is separately continuous.
\end{proof}

\bibliographystyle{amsplain}

\end{document}